\documentclass{article}
\usepackage[utf8]{inputenc}
\usepackage[T1]{fontenc}
\usepackage{graphicx}
\usepackage{amsmath}
\usepackage{amssymb}
\usepackage{amsthm}
\usepackage[english]{babel}
\usepackage[colorlinks]{hyperref}
\usepackage{url}

\newtheorem{Theorem}{Theorem}[subsection]
\newtheorem{Definition}{Definition}[subsection]
\newtheorem{Lemma}{Lemma}[subsection]

\title{Investigations on the properties of the arithmetic derivative}

\author{Niklas Dahl, Jonas Olsson, Alexander Loiko\footnote{alexandreloiko@gmail.com}}

\date{}

\begin{document}

\maketitle

\begin{abstract}
We investigate the properties of arithmetic differentiation, an attempt to adapt the notion of differentiation to the integers by preserving the Leibniz rule, $(ab)'=a'b+ab'$. This has proved to be a very rich topic with many different aspects and implications to other fields of mathematics and specifically to various unproven conjectures in additive prime number theory. Our paper consists of a self-contained introduction to the topic, along with a couple of new theorems, several of them related to arithmetic differentiation of rational numbers, a topic almost unexplored until now.
\end{abstract}

\tableofcontents

\section{An arithmetic derivative}
The arithmetic derivative function, from here and throughout the entire  text denoted by $n'$,  is a function $n': \mathbb{N} \rightarrow \mathbb{N}$ defined recursively by 
\begin{Definition}
\label{def:MEGA}
\begin{itemize}
    \item $p'=1$ for all prime numbers
    \item $(ab)'=a'b+ab'$ for all natural numbers $a,b$
\end{itemize}

\end{Definition}

We will begin by computing the arithmetic derivative (henceforth sometimes referred to as AD) for two interesting special cases.

\begin{Theorem}
  \[
     1' = 0
  \]
\end{Theorem}

\begin{proof}
  Using the Leibniz rule it is possible to prove that
  \[
     1=1^2 \Rightarrow 1' = (1^2)'\\
     \Leftrightarrow 1' = 1\cdot 1' + 1' \cdot 1 \\
     \Leftrightarrow 1' = 2\cdot 1'\\
     \Rightarrow 1' = 0
  \]
\end{proof}

\begin{Theorem}
  \[
    0'=0
  \]
\end{Theorem}

\begin{proof}
  This proof is similar to the previous one.

  \[ 0=2\cdot 0\Rightarrow \\
  0' = 2'\cdot 0 + 2\cdot 0' \Leftrightarrow \\
  0' = 2\cdot 0' \Rightarrow\\
  0'= 0
  \]
  
\end{proof}

We will shortly prove that $n'$ is well-defined. The proof depends on the following theorem.

\begin{Theorem} 
\label{log}
The solutions of the functional equation 

\[L:\mathbb{N}\rightarrow\mathcal{S}\]
\begin{equation}
  \label{eq:ekv1}
\ L(a)+L(b)=L(ab)
\end{equation}

in which   $S$ is an arbitrary ring under the usual operations $+$  and $\cdot$ are given by

 \begin{equation}\label{soln1}
    L(n)=\sum_{i=1}^{k}\alpha_{i}f(p_{i})
\end{equation}

where $\prod_{i=1}^{k}p_{i}^{\alpha_{i}}=n$ is the canonical prime factorization of $n$  and $f:\mathbb{P} \rightarrow \mathcal{S}$ is any function from the set $\mathbb{P}$ of all primes to $\mathcal{S}$.
\end{Theorem}

\begin{proof}
First we show that all solutions to (\ref{eq:ekv1}) are of the form (\ref{soln1}).
It follows by induction that
 
 \[L\left(\prod_{i=1}^{k}a_{i}\right)=\sum_{i=1}^{k}L(a_{i})\]
 
and 
 
 \[L\left(a^{b}\right)=bL(a)\]

Let $n=\prod_{i=1}^{k}p_{i}^{\alpha_{i}}$ be an arbitrary integer. We find that

 \[L(n)=\sum_{i=1}^{k}\alpha_{i}L(p_{i})\]

We define $f:\mathbb{P} \rightarrow \mathcal{S}$ as $f(p)=L(p)$ for every prime $p$. Then 

\[L(n)=\sum_{i=1}^{k}\alpha_{i}f(p_{i})\]

Next we prove that for every function $f:\mathbb{P}\rightarrow\mathcal{S}$
is $L(n)=\sum_{i=1}^{k}\alpha_{i}f(p_{i})$ a solution to (\ref{eq:ekv1})

We let $a=\prod_{i=1}^{k}p_{i}^{\alpha_{i}},\ b=\prod_{i=1}^{k}p_{i}^{\beta_{i}}$ in (\ref{eq:ekv1})

\begin{align*}
  \label{eq:bevis}
  \text{LHS}=L(a) + L(b) = 
  L\left(
  \prod_{i=1}^{k}p_{i}^{\alpha_{i}}
\right)+
L\left(
  \prod_{i=1}^{k}p_{i}^{\beta_{i}}
\right)
&=\\
\sum_{i=1}^{k}
\left(
  \alpha_{i}+\beta_{i}
\right)
f(p_{i})=
L\left(
  \prod_{i=1}^{k} p_i^{\alpha_i+\beta_i}
\right)&=\text{RHS}  
\end{align*}

\end{proof}

\begin{Definition} 
\label{def:logdef}
We call $f$ the \textit{prime} function of $L$. A solution $L$ to (\ref{eq:ekv1}) in \ref{log} we call an \textbf{arithmetically logarithmic} function.
\end{Definition}

\begin{Theorem}
\label{derivdefine}
The derivative $n'$ defined  in (\ref{def:MEGA}) 
is well-defined.
\end{Theorem}

\begin{proof}

Let $ld(n)=\frac{n'}{n}$ then the conditions on $ld$ are:
\begin{itemize}
\item $ld(p)=\frac{1}{p}$ for all prime numbers p
\item $ab\cdot ld(ab)=ab\cdot ld(a)+ab\cdot ld(b)\Leftrightarrow ld(a)+ld(b)=ld(ab)$
\end{itemize}
According to \ref{def:logdef} $ld$ is an arithmetically logarithmic function with the prime function $f(p)=\frac{1}{p}$. Then, by \ref{log}, it is well defined and can be written as

\[ld\left(\prod_{i=1}^{k}p_{i}^{\alpha_{i}}\right)=\sum_{i=1}^{k}\frac{\alpha_{i}}{p_{i}}\]
 
 or $n'=n\sum_{i=1}^{k}\frac{\alpha_{i}}{p_{i}}$
\end{proof}

\section{General properties of the derivative}
\subsection{Inequalities}



Here we will present some general properties of the arithmetic derivative.
All of these theorems were originally proved in \cite{bar} and are presented here for two reasons: we will use several of the theorems and definitions later and the proofs provide interesting examples of previous work in the field.
\begin{Theorem}\label{thm:ub}
  Let $n$ be a natural number and $k$ be the smallest prime factor in $n$. Then
  or every natural number $n$, 
\[
\frac{n\cdot \log_k{n}}{k}\geq n'
\]
with equality iff $n$ is a power of $k$.
\end{Theorem}

\begin{proof}
If
 \[n = \prod_{i=1}^mp_i^{a_i}\]
is the unique prime factorization of $n$, then, according to (\ref{derivdefine})
\[
n' = n \cdot \sum_{i=1}^m \frac {\alpha_i}{p_i} \leq n\sum_{i=1}^m \frac{\alpha_i}{k}
\]
\[ = n \sum \frac{1}{k}\] where the last sum iterates from one to the sum of all $\alpha_i$.
The last expression is not greater than
 \[ n \cdot \frac{1}{k} \log_k{n}\] because $\sum_{i=1}^{m}\alpha_i\leq \log_k{n}$ with equality iff $n$ is a perfect power of $k$.

\end{proof}

 \begin{Theorem}\label{thm:lb}
   For every natural non-prime $n$ with $k$ prime factors,
\[ n' \geq kn^{\frac{k-1}{k}}\]
 \end{Theorem}

 \begin{proof}
   If
   \[n = \prod_{i=1}^kp_i\]
 is the unique prime factorization of $n$, where a prime factor may appear several times, then
\[n' = n\sum_{i=1}^k\frac{p_i'}{p_i} = n\sum_{i=1}^k\frac{1}{p_i}\geq nk
\left(
  \prod_{i=1}^k\frac{1}{p_i}
\right)^{1/k} = n' \geq kn^{\frac{k-1}{k}}\]
according to the AG inequality.

 \end{proof}

\begin{Theorem}
The arithmetic derivative is uniquely defined over the integers by the rule $(-x)' = -(x')$
\end{Theorem}

\begin{proof}

First, we attempt to find the derivative of $-1$. After observing that $(-1)^2 = 1$, this is easy.
$(-1)^2 = 1 \rightarrow ((-1)^2)' = 1' \leftrightarrow 2 \cdot (-1) \cdot  (-1)' = 0$ (according to the Leibniz rule) $\leftrightarrow (-1)' = 0$.

Now we can use this new knowledge to derive any negative integer.
For every positive $k, (-k)' = ((-1) k)' = (-1)'k + (-1) k' = 0\cdot k - (k') = -(k')$

$(-k)' = -(k')$ for every integer k, or in other words, the arithmetic derivative is an odd function.
\end{proof}

\begin{Theorem}\label{thm:ratDef}
If we wish to preserve the Leibniz rule, then the arithmetic derivative is uniquely defined over the rational numbers by the rule $(a/b)' = (a'b-b'a)/b^2$.
\end{Theorem}

\begin{proof}

If we wish to preserve the Leibniz rule, then 1' must be equal to 0. From this we get the following equality for every non-zero integer n. 

$(n/n)' = 0 \Leftrightarrow n' \cdot (1/n) + n \cdot (1/n)' = 0 \Leftrightarrow (1/n)' = - n' / n^2$. Now we can show that $(a/b)' = a' \cdot (1/b) + a \cdot (1/b)' = \frac{a'/b - ab'}{b^2p} = \frac{a'b-ab' }{b^2}$.

Now we will prove that this formula is well-defined. It is sufficient to show that $\frac{ac}{bc}' = \frac{a}{b}'.$ 

\[\left(\frac{ac}{bc}\right)^{'} = \frac{(ac)'bc - ac(bc)'}{(bc)^2} = \frac{(a'c + ac')bc - ac(b'c+ bc')}{(bc)^2} = \]
\[ \frac{c^2 (a'b-ab')}{(bc)^2} = \frac{a'b-b'a}{b^2}=\left(\frac{a}{b}\right)^{'}\]
\end{proof}\newpage

\section{Further properties of the derivative}
\subsection{The rational derivative is unbounded}

It would be interesting to find a upper and lower bound for $n'$  like the ones described in (\ref{thm:ub})  and (\ref{thm:lb})  when $n$ is an arbitrary rational number.

\begin{Definition}
\[P(a,b) =
\left\{
  \begin{array}{ll}
    \text{True}  & \mbox{if \ } \ \forall \ L  \in  \mathbb{Q} \ \exists \ x  \in (a,b): \  x' \geq L \\
    \text{False}& \mbox{else}
  \end{array}
\right. \]
\end{Definition}

Or more simply that the function is true when for arbitrarily large  $L$ the rational interval $(a,b)$ contains another rational  number which when differentiated is not smaller than $L$. With this definition made we will address the following theorem.

\begin{Theorem}
  \textit{In any rational interval there exists a rational number with arbitrary large or small derivative.}
\end{Theorem}


\begin{proof}
This proof is rather long and depends on several lemmas.
\begin{Lemma}\label{lm:halv_ett}
\[P\left(\frac{1}{2},1\right) \text{is True} \]
\end{Lemma}

\begin{proof}
   We construct a sequence $\left\{a_i=\frac{2^i}{p_i}\right\}_{i=2}^{\infty}$  where $p_i$ is the smallest prime between $2^{i-1}$ and $2^i$. Such a $p_i$ always exists according to Bertrand's postulate. Observe the sequence of all numbers $a_i'$. By the rules of arithmetical differentiation (\ref{thm:ratDef}), $$a_i' = \left(\frac{2^i}{p_i}\right)^{'} = \left(\frac{2^{i-1}\cdot i}{p_i}-\frac{2^i}{p_i^2}\right)$$
since $2^{i-1}<p_i<2^i, $ we easily find that $a_i^{'} > \left( \frac{i}{2} - \frac {1}{2^{i-2}}\right)$ which obviously becomes arbitrary large as $i$ increases. All $a_i$'s lies between $1/2$ and $1$, so our proof is complete.
\end{proof}

\begin{Lemma}  \label{lm:ab_pil_kab}
  $P(a,b)\Rightarrow P(ka,kb)$ for positive rationals $a,b$ and $k$.
\end{Lemma}

\begin{proof}
  We need to prove that for all $N$, there are numbers in $(ka,kb)$ with derivative $\geq N$. We choose a rational $c\in (a,b)$ with $c'\geq \frac {N-k'a}{k}$ (such a $c$ always exists according to the definition of $P$). It is evident that $ka<kc<kb$. By the rules of differentiation we have that $(kc)'=k'c+c'k\geq k'c + N - k'a \geq N$ from the inequality on $c'$ and because $c>a$.
\end{proof}

\begin{Lemma}\label{lm:a2a}
  $P(a,2a)$ holds.
\end{Lemma}

\begin{proof}
  This follows directly from  (\ref{lm:halv_ett}) and (\ref{lm:ab_pil_kab}).
\end{proof}

\begin{Lemma}\label{lm:aplus1}
  $P(a,a+1)$ is true for all positive rationals $a$.
\end{Lemma}\newpage

\begin{proof}
  We prove this by contradiction.
   Assume that $P(a,a+1)$ is false for some $a$. 
   Then it follows from (\ref{lm:a2a}) that $P(a+1,2a)$ is true ((\ref{lm:a2a}) basically says that between $a$ and $2a$ there are numbers with large derivatives. 
   The assumption says that these numbers are not in $(a,a+1)$). 
   By using (\ref{lm:ab_pil_kab}) with $k=\frac{a}{a+1}$ we know that $P\left(a,2a\cdot\frac{a}{a+1}\right)$ is true. 
   Inductively repeating this procedure shows that $P\left(a,2a\cdot\left(\frac{a}{a+1}\right)^n\right)$ is also true. 
   We did earlier assume that $P(a,a+1)$ was not. 
   That now leads to contradiction since $2a\cdot \left(\frac{a}{a+1}\right)^n< a+1$ for sufficiently large values of $n$ (remember that $a$ is positive so $0<\frac{a}{a+1}<1$ and $r^n\rightarrow 0$ as n goes to infinity for all $0<r<1$).  
But wait! It's not! 
Because of the fact that $2a\cdot\left(\frac{a}{a+1}\right)^n\rightarrow 0$ as $n$ grows large, it will eventually become less than $a$ and we can no longer use the 
\[[\textbf{not }P(a,a+1)\wedge P(a,B)]\Rightarrow P(a+1,B)\] 
argument.
 But if we prove that there exists a value of $n$ such that 
 $a<2a\cdot\left(\frac{a}{a+1}\right)^n<a+1$ 
 everything would be all right again. In fact it does.
 Let $n\in \mathbb{N}$ be the greatest number such that $a<2a\cdot\left(\frac{a}{a+1}\right)^n$. This means that 
 $$a\geq 2a\cdot\left(\frac{a}{a+1}\right)^{n+1}$$  
This is equivalent to 
$$\left(\frac{a+1}{a}\right)\cdot a \geq 2a\cdot \left(\frac{a}{a+1}\right)^n\Leftrightarrow
 a+1 \geq 2a\cdot \left(\frac{a}{a+1}\right)^n\ $$
 which is exactly what we wanted to prove.
  We have a contradiction and $P(a,a+1)$ is true for every positive rational $a$.
\end{proof}

\begin{Lemma}\label{lm:a_plus_c}
  $P(a,a+c)$ is true for all positive rational $a,c$. 
\end{Lemma}
\begin{proof}
  By lemma (\ref{lm:aplus1}), $U\left(\frac{a}{c},\frac{a}{c}+1\right)$ holds. By lemma (\ref{lm:ab_pil_kab}) with $k=c$, this gives us that  $P(a,a+c)$ is true.

\end{proof}

If we define $Q\left(a,b\right)$ to denote the boolean function ``there exists numbers in $(a,b)$ with arbitrary \textit{small} derivatives'', 
it can similarly be shown that corresponding versions of lemma (\ref{lm:halv_ett}),  (\ref{lm:ab_pil_kab}), (\ref{lm:a2a}),  (\ref{lm:aplus1}) and (\ref{lm:a_plus_c}) are also true for $Q$.
We encourage our readers to do this exercise.

\begin{Lemma}\label{lm:pq}
  \[P(a,b) \Leftrightarrow Q(-b,-a)\]
\end{Lemma}

\begin{proof}
  By $P(a,b)$ we know that for each $N$ there is a number $x$ in $(a,b)$ with derivative larger than $N$. Then $(-x)' \leq -N$ which leads to $Q(-b,-a)$ since $-x \in (-b,-a)$. The reverse is proven similarly.
\end{proof}

Using (\ref{lm:a_plus_c}) and (\ref{lm:pq})  it is possible to deduce $P$ and $Q$  is true for all $a,b$ such that $a < b$.

That is the end of the proof.

\end{proof}
\newpage


\subsection{Some properties of the  $\Lambda$ function}
\begin{Definition}
  For all natural numbers $n$, we define $\Lambda (n)$ as the smallest natural number $m$ less than or equal to $n$ such that $m' = \max(0',1',2'\ldots n')$.
\end{Definition}

\begin{Theorem}
$\Lambda(2^a)= 2^a$ for every positive natural $a$.
\end{Theorem}

\begin{proof}
  According to theorem (\ref{thm:ub}), $n'\leq \frac{n\log_2n}{2}$ with equality iff $n$ is a perfect power of $2$. 
  This means that all smaller natural numbers will have a smaller derivative, thereby proving this theorem.
\end{proof}

\begin{Theorem}
  For every natural number $m$ there exists a natural number $N$ such that for every $n\geq N$, $2^m|\Lambda(n)$.
\end{Theorem}

\begin{proof}
 We prove this by contradiction.
 We assume that there exists an $m$ such that for every $N$ there exists an  $n>N$ such that $2^{m}\nmid\Lambda(n)$ and $\Lambda(n)=n$.

 We write $n = 2^a \cdot B$ where $B$ is odd and, by assumption, $a<m$.
 According to the rules of arithmetic differentiation,  

 \begin{eqnarray*}
  n'=&a2^{a-1}B+2^{a}B'\\
  \leq&a2^{a-1}B+2^{a}\frac{B\log_{3}B}{3}
  \end{eqnarray*}

  The inequality is valid because of theorem (\ref{thm:ub}) and the fact that the smallest prime factor in $B$ is at least $3$ (since $B$ is odd).

  \begin{eqnarray*}
  =&a2^{a-1}\frac{n}{2^{a}}+2^{a}\frac{\frac{n}{2^{a}}\log_{3}\frac{n}{2^{a}}}{3}\\
  =&\frac{an}{2}+\frac{n\log_{3}(n/2^{a})}{3}\\
  =&n\left(\frac{a}{2}+\frac{\log_{3}(n)-a\cdot\log_{3}2}{3}\right)\\
  =&n\left(a\left(\frac{1}{2}-\frac{\log_{3}2}{3}\right)+\frac{\log_{3}n}{3}\right)\\
  <&n\left(m\left(\frac{1}{2}-\frac{\log_{3}2}{3}\right)+\frac{\log_{3}n}{3}\right)
 \end{eqnarray*}

 The last inequality is true since $a<m$. Now let $f(n)$ be the last expression minus
 $\left(
   2^{\lfloor \log_2 n \rfloor}
 \right)^{'}=
 \lfloor\log_{2}n\rfloor2^{\lfloor\log_{2}n\rfloor-1}
 $ or 
 \[f(n) = n\left(m\left(\frac{1}{2}-\frac{\log_{3}2}{3}\right)+\frac{\log_{3}n}{3}\right) -\lfloor\log_{2}n\rfloor2^{\lfloor\log_{2}n\rfloor-1}\]
 If we can prove that $f(n)$ will always assume negative values for sufficiently large $n$, we are done. 
We will prove the stronger 
\[\lim_{n\rightarrow +\infty}f(n) = -\infty\]

Now we repeatedly apply floor inequalities and logarithm rules:
$\lfloor x\rfloor > x-1$.
\begin{eqnarray*}
  \lim_{n\rightarrow +\infty}f(n) =&\lim_{n\rightarrow +\infty} n\left(m\left(\frac{1}{2}-\frac{\log_{3}2}{3}\right)+\frac{\log_{3}n}{3}\right) -\frac{\lfloor\log_{2}n\rfloor2^{\lfloor\log_{2}n\rfloor}}{2}\\
  <& \lim_{n\rightarrow +\infty} n\left(
      m\left(
        \frac{1}{2}-\frac{\log_{3}2}{3}
      \right)
      +\frac{\log_{3}n}{3}
  \right) -
  \frac{
    \lfloor\log_{2}n\rfloor2^{\log_{2}n}
  }{4}\\
  =& \lim_{n\rightarrow +\infty} n\left(m\left(\frac{1}{2}-\frac{\log_{3}2}{3}\right)+\frac{\log_{3}n}{3}\right) -\frac{n\lfloor\log_{2}n\rfloor }{4}\\
  <& \lim_{n\rightarrow +\infty} n\left(m\left(\frac{1}{2}-\frac{\log_{3}2}{3}\right)+\frac{\log_{3}n}{3}\right) -\frac{n (\log_2(n) -1) }{4}\\
  =& \lim_{n\rightarrow +\infty} n\left(m\left(\frac{1}{2}-\frac{\log_{3}2}{3}\right)+\frac{1}{4}+\frac{\log_{3}n}{3}-\frac{\log_2n}{4}\right)
\end{eqnarray*}

If we can prove that the expression inside the parenthesis becomes negative as $n\to\infty$ we are done. 
If 
\[\lim_{n\to\infty}\left(\frac{\log_{3}n}{3}-\frac{\log_2n}{4}\right) = -\infty\] this is obviously true.
Note that \[\frac{\log_{3}n}{3}-\frac{\log_2n}{4} = \log_2n\left(\frac{1}{3\log_23}-\frac{1}{4}\right) \approx \log_2n\cdot (-0.0396901)\] according to the logarithm laws.
This means that the entire expression becomes negative as $n\to\infty$. 
But this means that 
\[\lim_{n\to\infty}f(n)=-\infty\]
and gives us that $\left(2^{\lfloor \log_2 n\rfloor}\right)^{'}>n'$ for sufficiently large $n$ satisfying the assumptions, which contradicts the assumption that $\Lambda(n)=n$.
This ends the proof.
\end{proof}


\end{document}